\documentclass[12pt]{amsart}
\usepackage[centertags]{amsmath}
\usepackage{amsfonts}
\usepackage{amssymb}
\usepackage[latin5]{inputenc}
\usepackage{amsthm}
\usepackage{mathrsfs}
\usepackage{upgreek}
\usepackage{amsopn}
\usepackage{amscd}
\newtheorem{theorem}{Theorem}[section]
\newtheorem{lemma}[theorem]{Lemma}
\newtheorem{definition}[theorem]{Definition}

\theoremstyle{definition}

\newtheorem{rem}[theorem]{Remark}

\theoremstyle{remark}

\numberwithin{equation}{section}

\renewcommand{\leq}{\leqslant}
\renewcommand{\geq}{\geqslant}

\renewcommand{\epsilon}{\varepsilon}

\begin{document}
{\setlength{\baselineskip}%
                        {1.3\baselineskip}
                      
\title{All topologies come from a family of $0-1$-valued quasimetrics}
\author{Zafer Ercan and Mehmet Vural}
\address{(Zafer Ercan \& Mehmet Vural) Department of Mathematics, Abant \.{I}zzet Baysal University, G\"{o}lk\"{o}y Kamp\"{u}s\"{u}, 14280 Bolu, Turkey}
\email{zercan@ibu.edu.tr \& m.vural.hty@gmail.com}

\maketitle 
\begin{abstract} We prove the statement in the title.  
\end{abstract}
\section{Continuity spaces and Kopperman's Theorem} 
Topological spaces are natural extensions of metric topologies. A topological space whose topology is a metric topology is called a {\bf\emph{metrizable space}}. Most of the fundamental examples of topological spaces are not metrizable (for general definitions and examples, see \cite{eng}), thereby one of the fundamental research topics in General Topology has been to find conditions under which a topological space becomes metrizable. In \cite{kopperman}, despite the fact that all topologies taken into account are not metrizable, types of such conditions are shown to be obtained in terms of extended quasi-metrics. To prove this, Kopperman introduced the notion of continuity spaces in \cite{kopperman}, which reads as follows.
 
A semigroup $A$ with identity and absorbing element ${\infty}\not =0$ is called a {\bf\emph{value semigroup}} if the following conditions are satisfied:
 \begin{enumerate}
\item[\textnormal{(i)}] If $a+x=b$ and $b+y=a$, then $a=b$ (in this case, if $a\leq b$ is defined as $b=a+x$ for some $x$, then $\leq$ defines a partial order on $A$).
\item[\textnormal{(ii)}] For each $a$, there is a unique $b$ such that $b+b=a$ (in this case, one writes $b={1\over 2}a$).
\item[\textnormal{(iii)}] For each $a$, $b$, the element $a\wedge b:=\inf\{a,b\}$ exists.
\item[\textnormal{(iv)}] For each $a$, $b$, $c$, the equality $a\wedge b+c=(a+c)\wedge (b+c)$ holds.
\end{enumerate}
A {\bf\emph{set of positives}} in a value semigroup $A$ is a subset $P\subset A$ satisfying the following:
 \begin{enumerate}
\item[\textnormal{(i)}] if $a$, $b\in P$, then $a\wedge b\in P$; 
\item[\textnormal{(ii)}] $r\in P$ and $r\leq a$, then $a\in P$;
\item[\textnormal{(iii)}] $r\in P$, then ${r\over 2}\in P$;
\item[\textnormal{(iv)}] if $a\leq b+r$ for each $r\in P$, then $a\leq b$.
\end{enumerate}
Let $X$ be a non-empty set, $A$ a value semigruop, $P$ a set of positives of $A$, and $d:X\times X\rightarrow A$ a function such that 
$d(x,x)=0$ and $d(x,z)\leq d(x,y)+d(y,z)$ for all $x$, $y$, $z\in A$. Then $\mathcal{A}=(X,d,A,P)$ is called a {\bf\emph{continuity space}}.
For each $x\in X$ and $r\in P$, we write
\begin{center}$B[x,r]=\{y\in X:d(x,y)\leq r\}$\end{center}
\begin{theorem} [Kopperman, \cite{kopperman}] Let $\mathcal{A}=(X,d,A,P)$ be a continuity space. Then 
\begin{center}$\textnormal{To}(\mathcal{A}):=\{U\subset X:\mbox{ for each } x\in U\mbox{ there exists } r\in P\mbox{ such that }B[x,r]\subset U\}$\end{center}
is a topology on $X$. Moreover, every topology on $X$ is of this form.
\end{theorem}
\section{The Main Result}
The main issue of the present note is tto reveal the fact that  Kopperman's theorem can be refined by taking $0-1$-valued generalized quasi-metric spaces insistead of continuity spaces. We will first define related notions which will be used in the sequel.

A function $d:X\times X\rightarrow [0,{\infty})$ is called a {\bf\emph{quasi-metric}} if $d(x,x)=0$ and $d(x,z)\leq d(x,y)+d(y,z)$ for all $x$, $y$, $z\in X$. A {\bf\emph{$0-1$-valued generalized quasi-metric}} on a set $X$ is a function from $X\times X$ into $\{0,1\}^{I}$ for some non-empty set $I$
if for each $i\in I$ the function $d_{i}:X\times X\rightarrow\{0,1\}$ defined by $d_{i}(x,y)=d(x,y)(i)$ is a quasi-metric. In such a case, we will refer to $(d_{i})_{i\in I}$ as a 
{\bf\emph{partition}} of $d$.
A set  $X$ equipped with a $0-1$-valued generalized quasi-metric  $d$ is called a {\bf\emph{$0-1$-valued generalized quasi-metric space}}.
Following the usual custom, we denote this space by $(X,d,I)$. A subset $U\subset X$ is called {\bf\emph{open}} if for each $x\in U$ there exists a finite 
 set $J\subset I$ such that 
$$\bigcap_{i\in J}\{y\in X:d(x,y)(i)=0\}\subset U.$$ The set of open sets with respect to $(X,d,I)$ is denoted by 
$\textnormal{To}(X,d,I)$. 
\begin{lemma} Let $(X,d,I)$ be a $0-1$-valued generalized quasi-metric space. Then, for each $x\in X$ and $i\in I$, the set $\{y\in X:d_{i}(x,y)=0\}$ is open.
\end{lemma}
\begin{proof} Let $U:=\{y\in X:d_{i}(x,y)=0\}$ and let $y\in U$ be given. Then $d_{i}(x,y)=0$. If $d_{i}(y,z)=0$, then we have 
$$0\leq d_{i}(x,z)\leq d_{i}(x,y)+d_{i}(y,z)=0,$$ so that
$$\{z:d_{i}(y,z)=0\}\subset U.$$ It follows that $U$ is open. 
\end{proof}
The proof of the following is elementary and is therefore omitted.
\begin{theorem} Let $(X,d,I)$ be a $0-1$-valued generalized quasi-metric space. Then $\textnormal{To}(X,d,I)$ is a topological space. If $(d_{i})_{i\in I}$ is the partition corresponding to $d$, then the familyy
$$\{\{y\in X:d_{i}(x,y)=0\}:i\in I, x\in X\}$$ is a subbase of $\textnormal{To}(X,d,I)$.
\end{theorem}

Let us denote a value of a proposition $p$ by $t(p)$. Let $(X,{\tau})$ be a topological space and $U\in{\tau}$ be given. For each $U\in{\tau}$, the map $d_{U}:X\times X\rightarrow\mathbb{R}$ defined by

$$\begin{cases} 0, & \textnormal{if}\ \ t(x\in U\implies y\in U)=1;\\ 1, &\textnormal{if}\ \ t(x\in U\implies y\in U)=0. \end{cases}$$
is a quasi-metric. Indeed, let $x$, $y$ and $z\in X$ be given. If $x\not\in U$ then, $t(x\in U\implies y\in U)=1$ so $d_{U}(x,z)=0$. If $x\in U$,
$d_{U}(x,y)=0$ and $d_{U}(y,z)=0$, then $y\in U$ and $z\in U$. Thus $t(x\in U\implies z\in U)=1$, so $d_{U}(x,z)=0$. This shows that $d_{U}$ is a quasi-metric. Also, for each $u \in \tau$, one can define a function $p_{u}: X \times X \rightarrow \mathbb{R} $ as $p_{u}(x,y)= \chi_{\{u\}}(x) \chi_{\{u^{c}\}}(y)$ is also a quasi-metric,which is equivalent to the quasi-metric $d_{u}$.
In particular, we have the following.
\begin{lemma} Let $(X,{\tau})$ be a topological space and $U\in{\tau}$ be given. Then, for each $x\in U$, one has
$$U=\{y\in X:d_{U}(x,y)=0\}.$$
\end{lemma}
Interestingly enought, the converse of the above fact is also true.
\begin{theorem} Every topological spaces comes from a $0-1$-valued generalized quasi-metric space. That is, if $(X,{\tau})$ is a topological space, then there exists a $0-1$-valued generalized quasi-metric on $X$ such that ${\tau}=\textnormal{To}(X,d,I)$.
\end{theorem}
\begin{proof} For each $x$, $y\in X$ and $U\in{\tau}$, if the proposition ``$x\in U\implies y\in U$'' is true let $d(x,y)(U)=0$, and otherwise let it be 1. Then we have a function $d:X\times X\rightarrow\{0,1\}^{\tau}$. One can easily show that it is indeed a $0-1$-valued generalized quasi-metric. Now we show that ${\tau}=\textnormal{To}(X,d,I)$. Let $U\in{\tau}$ and $x\in U$ be given. 
Since $\{y\in X:d_{U}(x,y)=0\}\in \textnormal{To}(X,d,I)$ it directly follows that
$$U=\{y\in X:d_{U}(x,y)=0\},$$ whence $U\in To(X,d,I)$. Now, let $V\in To(X,d,I)$ be given. 
If $V=X$, then obviously $V\in {\tau}$. Suppose that $V\not =X$. 
Let $x\in V$. Then there exitst $U_{1},\ldots U_{n}\in {\tau}$ such that
$$\bigcap_{i=1}^{n}\{y\in X:d_{U_{i}}(x,y)=0\}\subset V.$$
By Lemma 2.3 we have 
$$x\in\bigcap_{i=1}^{n}U_{i}\subset V,$$ 
so that $V\in{\tau}$. The proof that ${\tau}=To(X,d,I)$ is now complete.
\end{proof}
It is obvious that a subbase of the space $To(X,d,I)$ is 
\begin{center}$\mathcal{B}=\{\{y:d_{i}(x,y)=0\}:x\in X,i\in I\}$,\end{center} where $(d_{i})_{i\in I}$ is a partition of $d$.

Through lack of symmetry, categorizing the notion of convergence as right convergence and left convergence is reasonable in a $0-1$-valued generalized quasi-metric space $(X,d,I)$. The definition is as follows.
\begin{definition}
 A net $(x_\alpha)_{\alpha\in A} $ right converges to $x$ in $(X,d,I)$, denoted by $(x_\alpha)\xrightarrow{r} x$, if for each $i \in I$ there exists $\alpha_{0} \in A $ such that $d_{i}(x,x_{\alpha})= 0$ for all $\alpha \geq \alpha_{0}$ . A net $(x_\alpha)_{\alpha\in A} $ is called right Cauchy (or, $r$-Cauchy) if for each $i \in I$ there exist $\alpha_{0} \in A $ such that $d_{i}(x_{\alpha},x_{\beta})= 0$ for all $\beta \geq \alpha \geq \alpha_{0}$.
\end{definition}
The definitions of left convergence and left Cauchyness are given similarly: for the sake of simplicity, only `right' versions of them are used in the rest of the note.
\begin{rem} Several familiar topological notions can be derived using the structure of $0-1$-valued generalized quasi-metric spaces. We list some of them below.

\begin{enumerate}
\item[(1)] Let $(X,d,I)$ be a $0-1$-valued generalized quasi-metric space. Then the following are equivalent:
\begin{enumerate}
\item[(a)] The net $(x_{\alpha})_{\alpha \in A}$ right converges to $x$ in $(X,d,I)$.
\item[(b)] The net $(x_{\alpha})_{\alpha \in A}$ converges to $x$ in $\textnormal{To}(X,d,I)$.
\item[(c)] The net $d(x,x_{\alpha})$ converges to zero in the product topological space ${\{0,1\}}^{I}$.
\end{enumerate}
\item[(2)] Let $(X,d,I)$ and $(Y,p,J)$  be $0-1$-valued generalized quasi-metric spaces, and $f$ a function from $X$ into $Y$. Then $f$ is continous at a point if and only if for each $j \in J$ there exists $i \in I$ such that $d_{i}(x,y)=0$ implies $p_{j}(f(x),f(y))=0$.
\item[(3)] Let $(X,d,I)$ be a $0-1$-valued generalized quasi-metric space. Then $\textnormal{To}(X,d,I)$ is a $T_0$ space if and only if for every distinct pair $x,y \in X$ there exists $i\in I$ such that $d_{i}(x,y)=1$.
\item[(4)]  $\textnormal{To}(X,d,I)$ is $T_{1}$-space if and only if for every distinct pair $x,y \in X$ there exists $i,j \in I$ such that $d_{i}(x,y)=1$ and $d_{i}(y,x)=1$.
\item[(5)] $\textnormal{To}(X,d,I)$ is a $T_{2}$-space if and only if for every distinct pair $x,y \in X$ there exists $i,j \in I$ such that $$d_{i}(x,y)=d_{i}(y,x)=d_{j}(x,y)=d_{j}(y,x)=1.$$

\item[(6)] The notion of statistical convergence of a sequence of real numbers is as follows: A sequence $(x_{n})$ of real numbers is said to {\bf\emph{converge statistically}} to the real number $x$ if for each $\epsilon>0$ one has $\delta(A_{\epsilon})=0$, where $A_{\epsilon}=\{n \in \mathbb{N} : |x_{n}-x|\geq \epsilon\}$ and $$\delta(A_{\epsilon})=\lim_{n \rightarrow \infty} {\sum_{\substack {a \in A_{\epsilon}, a\leq n}}1\over n}.$$ In \cite{kocina}, it is defined for topological spaces as well. Here is its variant using the aforementioned arguments: A sequence $(x_{n})$ in $(X,d,I)$ is said to convergence statistically to $x$ if $$\lim_{n \rightarrow \infty} \frac{|\{k\leq n:d_{i}(x,x_{k})= 1\}|}{n}=0$$ holds for each $i \in I$.
\end{enumerate}
\end{rem}

\end{document}